\documentclass[12pt]{article}

\usepackage{amssymb}
\usepackage{amsmath}

\newtheorem{theorem}{Theorem}[section]

\newtheorem{corollary}[theorem]{Corollary}
\newtheorem{remark}[theorem]{Remark}

\newenvironment{proof}{\noindent {\bf
Proof}.\ }{\proofbox\par\smallskip\par}

\def\numberlikeadb{\global\def\theequation{\thesection.\arabic{equation}}}
\numberlikeadb

\newcommand{\halmos}{\rule{1ex}{1.4ex}}
\newcommand{\proofbox}{\hspace*{\fill}\mbox{$\halmos$}}

\newcommand{\pr}{{\mathbb P}}

\newcommand{\reals}{{R}}
\newcommand{\ints}{\mathbb{Z}}

\newcommand{\re}{\mbox{\bf Re}}

\newcommand{\eqa}{\begin{eqnarray}}
\newcommand{\ena}{\end{eqnarray}}
\newcommand{\eq}{\begin{equation}}
\newcommand{\en}{\end{equation}}
\newcommand{\eqs}{\begin{eqnarray*}}
\newcommand{\ens}{\end{eqnarray*}}

\def\zz{\ints}
\def\l{\lambda}

\def\a{\alpha}
\def\b{\beta}
\def\g{\gamma}

\def\d{\delta}
\def\D{\Delta}
\def\e{\varepsilon}
\def\h{\eta}
\def\z{\zeta}
\def\th{\theta}

\def\k{\kappa}
\def\m{\mu}
\def\n{\nu}

\def\s{\sigma}

\def\t{\tau}
\def\f{\varphi}

\def\nin{\noindent}

\def\msk{\medskip}

\def\Blb{\left\{}
\def\Brb{\right\}}

\def\giv{\,|\,}

\def\non{\nonumber}

\def\Eq{\ =\ }

\def\Le{\ \le\ }

\def\sji{\sum_{j\ge1}}

\def\sio{\sum_{i\ge0}}

\def\Ref#1{{\rm (\ref{#1})}}

\def\bp{\begin{proof}}
\def\ep{\end{proof}}

\def\bone{{\bf 1}}

\def\bl{\lb}

\def\xx{{\cal X}}

\def\ignore#1{}
\def\ex{{\mathbb E}}

\def\tmax{t_{\max}}

\def\Z{\ints}

\def\R{\reals}

\def\jj{{\mathcal J}}

\def\sjj{\sum_{J\in \jj}}

\def\uis{^{[1]}}
\def\uts{^{[2]}}

\def\tx{{\tilde x}}

\def\ui{^{(1)}}

\def\ut{^{(2)}}
\def\uh{^{(3)}}
\def\uf{^{(4)}}

\def\bz{\bar\z}
\def\re{\mathbb{R}}

\def\bg{{\bar \g}}

\def\bd{{\bar \d}}
\def\bl{{\bar \l}}

\def\zz{{\cal Z}}

\def\Def{\ :=\ }

\def\sjZ{\sum_{z\in\zz}}
\def\nat{{\mathbb N}}

\def\slid{\sum_{l=1}^d}
\def\xx{{\cal X}}
\def\xxs{{\cal X}'}
\def\sjtZ{{\sum_{z \in \tzz}}}
\def\sjZi{{\sum_{\bj \in \zz_1}}}

\def\tzz{{\widetilde\zz}}
\def\hs{\hat\s}
\def\bm{{\bar m}}
\def\bi{{\mathbf i}}
\def\bj{{\mathbf j}}
\def\pp{{\mathcal P}}
\def\aa{{\mathcal A}}
\def\uK{{^{(K)}}}
\def\uk{{^{(k)}}}
\def\ukd{{^{(k')}}}
\def\yy{{\cal Y}}
\def\var{{\rm Var\,}}
\def\lti{\lim_{t\to\infty}}

\begin{document}

\title{Individual and patch behaviour in structured metapopulation models}
\author{
A. D. Barbour\footnote{Institut f\"ur Mathematik, Universit\"at Z\"urich,
Winterthurertrasse 190, CH-8057 Z\"URICH; email: {\tt a.d.barbour@math.uzh.ch};
work begun while ADB was Saw Swee Hock 
Professor of Statistics at the National University of Singapore, and
supported in part by Australian Research Council Grants Nos DP120102728 and DP120102398.\msk}
\ and
Malwina Luczak\footnote{School of Mathematical Sciences, QMUL, Mile End Road, London E1 4NS, UK; 
email: {\tt m.luczak@qmul.ac.uk}; supported by an EPSRC Leadership Fellowship, grant reference EP/J004022/2, 
and in part by Australian Research Council Grant No DP120102398.}
\\ \\
Universit\"at Z\"urich and Queen Mary University, London
}

\date{}
\maketitle

\begin{abstract}
Density dependent Markov population processes with countably many types
can often be well approximated over finite time intervals by the solution 
of the differential equations that describe their average drift, 
provided that the total population size is large. They also  
exhibit diffusive stochastic fluctuations on a smaller scale about this 
deterministic path.  Here, it is shown that the individuals in such processes experience
an almost deterministic environment.  Small groups of individuals
behave almost independently of one another, evolving as Markov jump
processes, whose transition rates are prescribed functions of time.
In the context of metapopulation models, we show that `individuals'
can represent either patches or the individuals that migrate among
the patches; in host--parasite systems, they can represent both hosts and
parasites.
\end{abstract}

 \noindent
{\it Keywords:}  Markov population processes, propagation of chaos, \hfil\break
{}\hglue1in metapopulation, host parasite systems \\
{\it AMS subject classification:} 92D30, 60J27, 60B12 \\
{\it Running head:} Structured metapopulation models

\section{Introduction}\label{introduction}
\setcounter{equation}{0}

In a series of papers motivated by models of structured metapopulations 
(Levins 1969, Hanski \& Gilpin 1991) and parasitic disease transmission (Kretzschmar 1993),
the authors have extended Kurtz's (1970, 1971) theory to provide
laws of large numbers and central limit theorems for Markov population
processes with countably many types of individual, together with estimates
of the approximation errors: see Barbour \& Luczak [BL]~(2008, 2012a,b).
These theorems provide a good description of the overall behaviour of
such processes, when the population size is large.  However, as observed 
by L\'eonard~(1990), many ecological models, when seen from the perspective 
of the individuals themselves, can be interpreted as interacting particle 
systems.  It is then of interest to be able to describe the behaviour of
(small groups of) individuals within the large system.  Under very stringent
assumptions on the transition rates, in particular requiring that they be
uniformly bounded, he proves a `propagation of chaos' theorem, showing
that individuals evolve almost independently of one another, as Markov
processes whose transition rates are determined by the bulk behaviour
of the system.  

In this paper, we establish an analogous result for
systems with countably many types, under much less restrictive conditions.
We formulate a model that is general enough to encompass many host
parasite systems and structured metapopulation models.
The main tool used in showing the asymptotic independence of individuals
in such processes is to couple the process describing
the evolution of individuals in the original system with one in which they
evolve independently.  The coupling is constructed by matching the transition 
rates in the two processes, and the argument is described in Section~\ref{main}.  

In order to show that the coupling is close,
we rely on the quantitative law of large numbers proved in [BL]~(2012a).
The conditions needed for the law of large numbers have already been shown to be
satisfied for a number of examples from the literature, 
including the models of Arrigoni~(2003), Barbour \& Kafetzaki~(1993), Kretzschmar~(1993)
and Luchsinger~(2001a,b).  However, some work
is required to find explicit conditions based on the parameters of our general model
under which the law holds;
this is accomplished in Section~\ref{LLN}.  The paper concludes with examples taken
from Metz \& Gyllenberg~(2001) and from Kretzschmar~(1993).

\section{Main results}\label{main}
\setcounter{equation}{0}

We begin by formulating our models in a way which explicitly reflects their origins
in metapopulation and parasitic disease modelling.  The basic description is in 
terms of the numbers of patches of each of a countable number of types.
The type of a patch is determined by the numbers of animals of each of~$d$ different
varieties present in the patch, indexed by $\bi = (i_1,\ldots,i_d) \in \Z_+^d$.   For instance,
a patch may represent a host, and its type the numbers of parasites of various different 
species that it harbours.  However, an animal's variety may also indicate its developmental
stage, or its infection status, so that its variety may change over its lifetime.  
We also define~$d$ further types, to account
for animals of the different varieties that are in transit between patches. Thus the possible patch types 
are indexed by $\zz := \zz_1 \cup \zz_2$, where $\zz_1 = \Z_+^d$ and 
$\zz_2 =\{1,\ldots,d\}$.  
In these terms, the state space is expressed as 
$\xx := \{X \in\Z_+^{\zz},\,\sjZ X_z < \infty\}$.  The interpretation is that 
$X_\bi$ records the number of patches of type~$\bi$,  $\bi\in\zz_1$,   whereas $X_{l}$, 
$1\le l\le d$, denotes the number of migrating animals of variety~$l$.
The restriction $\sjZ X_z < \infty$ in the definition of~$\xx$ constrains total numbers of patches
and animals to be finite.   Our model for the evolution of the metapopulation consists of a
family~$X^N := (X^N(t),\,t\ge0)$ of pure jump Markov processes on~$\xx$, indexed by $N \in \nat$,
with~$N$  to be thought of as a typical number of patches in the process~$X^N$.   Writing $e(z)$ for the 
$z$-coordinate vector in~$\re_+^{\zz}$, $z\in\zz$, and $e_{l}$ for the $l$-th coordinate vector in~$\Z^d$, 
the transition rates for~$X^N$ are assumed to be given by
$$
\begin{array}{rllllllr}
  {\rm I}:& \ X & \to &X + e(\bj) - e(\bi) & \\
    & &&\qquad\qquad\quad\ \
    \mbox{at rate}\quad X_\bi\{\bl_{\bi\bj} + \l_{\bi\bj}(x)\}, &\ \bi,\bj \in \zz_1; & \\[1ex]
  {\rm II}:& \  X & \to &X +  e(\bi)      
      \qquad\mbox{at rate}\quad N\b_\bi(x), &\ \bi \in \zz_1;  & \\[1ex]
  {\rm III}:& \  X & \to &X - e(\bi)       
    \qquad  \mbox{at rate}\quad X_\bi\{\bd_{\bi} + \d_{\bi}(x)\}, &\ \bi \in \zz_1; & \\[1ex]
   {\rm IV}:& \ X & \to &X + e(l) + e(\bi - e_l) - e(\bi) \\
     &&&\qquad\qquad\quad\ \ \mbox{at rate}\quad X_\bi\{\bg_{\bi l} + \g_{\bi l}(x)\},  
                                                       &\ \bi \in \zz_1,\,1\le l\le d; & \\[1ex]
   {\rm IV'}:& \ X & \to &X + e(l) 
       \qquad\mbox{at rate}\quad \sjZi X_\bj\{\bg'_{\bj l} + \g'_{\bi l}(x)\},  
                                                       &\ 1\le l\le d; & \\[1ex]
   {\rm V}:& \ X & \to &X + e(\bi+e_l) - e(\bi) - e(l)  \\
     &&&\qquad\qquad\quad\ \ \mbox{at rate}\quad X_{l} x_\bi\s_{l\bi}(x), &\ \bi \in \zz_1,\,1 \le l \le d, & \\[1ex]
   {\rm VI}:& \ X & \to &X - e(l)   
       \qquad  \mbox{at rate}\quad X_{l}\{\bz_l + \z_l(x)\}, &\ 1\le l\le d, &
\end{array}
$$
where $x := N^{-1}X \in \{x' \!\in\! \R_+^{\zz},\,\|x'\|_1 < \infty\} =: \xxs$, and $\|x\|_1 := \sjZ x_z$.

\newpage
The transitions~I correspond to changes in the type of a patch, because of births, deaths and changes of
status involving animals within the patch, or as a result of infection or catastrophe, 
or of immigration from outside the metapopulation, and we set 
$\bl_{\bi\bi} = \l_{\bi\bi}(\cdot) = 0$, $\bi\in\zz_1$.  Then II and~III correspond to the
creation and destruction of patches,  IV and~V concern the migration of animals
of the different varieties between patches, and~VI the deaths of animals during migration.
The transitions~IV$'$ allow for the possibility of an individual being born as a migrant,
as is allowed in our first example, in Section~\ref{examples}.
More complicated transitions of this kind could have been incorporated, but the biological
motivation for doing so does not seem compelling.  
The parameters $\bl_{\bi\bj}$, $\bd_\bi$, $\bg_{\bi l}$, $\bg'_{\bi l}$ and $\bz_l$ represent fixed
rates of transition per patch. To ensure that the overall rate of jumps is finite at
any $x  \in N^{-1}\xx$, it is necessary to have
$\sjZi \bl_{\bi\bj} < \infty$ for all~$\bi\in\zz_1$.  The corresponding quantities without the bars,
together with $\s_{l\bi}(\cdot)$ and~$\b_\bi(\cdot)$, represent state
dependent components of the transition rates.  For each $x \in \xxs$,
it is then also necessary to have
\eq\label{parameter-conds}
  \sjZi \l_{\bi\bj}(x) \ <\ \infty, \quad \sjZi \b_\bj(x) \ <\ \infty \quad\mbox{and}\quad
    \sjZi x_\bi\s_{l\bi}(x) \ <\ \infty; 
\en
further assumptions are added in Section~\ref{LLN}.
In transition~IV, we require $\bg_{\bi l} = \g_{\bi l}(x) = 0$ whenever $i_l = 0$, to avoid
ever having $i_l < 0$, which would be biologically meaningless.

Let $T > 0$ be a constant; we study the evolution of the metapopulation over the interval $[0,T]$. Under further assumptions on the transition rates I--VI and on the initial condition~$x^N(0)$, 
it can be shown that, with high probability, $x^N(t)$ is uniformly close to the solution~$x$ of 
a deterministic integral equation, which is the analogue of the usual
deterministic drift differential equations found in finite dimensional problems.
In Section~\ref{LLN}, we illustrate how to use the results of [BL]~(2012a) to justify this.  
For the rest of this section, we assume that
\eq\label{LLN-approx-2}
  \pr\Bigl[\sup_{0\le t\le T}\|x^N(t) - x(t)\| _\m > \e_N\Bigr] \Le P_T(N,\e_N),
\en 
for some (small) $\e_N$ and~$P_T(N,\e_N)$, and for some norm $\|\cdot\|_\m$,
and show how~\Ref{LLN-approx-2} can be used to establish the joint behaviour of groups of individuals
in the process~$X^N$.

We begin by investigating the behaviour over time of the type of a
single patch~$\pp$.  The transitions I, IV and~V each contain elements corresponding
to the rate of change of type of a patch that is currently of type~$\bi$, with the rates
depending on the current state of the whole system, and the death rate of such
a patch is given in~III.  Thus we can single out the transition rates for the
patch~$\pp$, with its evolution only being Markovian if the current state~$x$ of the whole
system is adjoined.  For any $\bi,\bj \in \zz_1$ and $1\le l\le d$, these take the form
\eq\label{patch-rates}
\begin{array}{rlllllr}
  \bi& \to &\bj 
    &\mbox{at rate}\quad \bl_{\bi\bj} + \l_{\bi\bj}(x), &\quad  \|\bj-\bi\|_1 \ge 2; & \\[0.5ex]
  \bi& \to &\bj  
    &\mbox{at rate}\quad \bl_{\bi\bj} + \l_{\bi\bj}(x) + \bg_{\bi l} + \g_{\bi l}(x); 
            &\quad \bj = \bi - e_l & \\[0.5ex]
  \bi& \to &\bj  
    &\mbox{at rate}\quad \bl_{\bi\bj} + \l_{\bi\bj}(x) + x_{l}\s_{l\bi}(x); &\quad \bj = \bi + e_l & \\[0.5ex]
  \bi& \to &\D 
    &\mbox{at rate}\quad \bd_\bi + \d_\bi(x),  &   
\end{array}
\en
with~$\D$ a state to represent that the patch has been destroyed.
We let~$Y_N$ denote the process describing the time evolution of the type assigned to~$\pp$,
with $Y_N(t)$ taking values in $\zz_1 \cup \D$;
the $N$-dependence reflects that its transition rates are as described in~\Ref{patch-rates}, 
but with~$x^N(t)$ in place of~$x$ for the rates at time~$t$.

Analogously, we could define a process representing the life history of an animal~$\aa$ in the
metapopulation.  The migration transitions IV, V and~VI are easy to interpret, and the
destruction of a patch in~III implies the death of any animals in that patch.   The transitions~I are
more complicated.  Considering an animal of variety~$l$, its death is typically recorded in a transition
in which $j_l \le i_l-1$ (several animals of the same variety may die as a result of the same event), but a
change of developmental stage, for instance, may also result in $j_l = i_l - 1$.
Then, for unicellular animals, division is recorded most simply as $j_l = i_l + 1$, though it may be
useful to interpret the same event as the death of the original animal at the same time as the
birth of two offspring. 
Furthermore, transitions in which $i_l$ does not change may represent births of
animals that are directly associated with the particular animal of variety~$l$ being considered, 
as when an adult gives birth to juveniles that are represented as a distinct variety;  such events
are naturally to be recorded in a life history.  This suggests defining a life history 
process~$Z_N := \{(Z_{N0}(t),\ldots,Z_{Nd}(t)),\,
t \ge 0\}$ for an animal~$\aa$,
whose statespace is 
$$
   ((\zz_1 \times \{1,2,\ldots,d\}) \cup \{1,2,\ldots,d\} \cup \D) \times \Z_+^d.
$$ 
A value $Z_{N0}(t) \in \zz_1 \times \{1,2,\ldots,d\}$ denotes the 
the type of patch in which~$\aa$ is living and its current variety. Then
$Z_{N0}(t) = l$ if~$\aa$ is of variety~$l$ and in migration, 
and, if $Z_{N0}(t) = \D$, the animal~$\aa$ has died before time~$t$. 
The values $Z_{Nl}(t)$, $1\le l\le d$, record the numbers of children
of the different varieties to which~$\aa$ has given birth up to time~$t$. 
For $\bi\in \zz_1$, $l,l'\in \{1,2,\ldots,d\}$ and $m,s\in \Z_+^d$,
the transition rates can be represented in the form
\eq\label{individual-rates}
\begin{array}{rlllllr}
  ((\bi,l),m) & \to &((\bi+s,l),m+s) 
    &\mbox{at rate}\quad \bl\ui_{\bi ls} + \l\ui_{\bi ls}(x) ; & \\[0.5ex] 
  ((\bi,l),m) & \to &((\bj,l),m)   
    &\mbox{at rate}\quad \bl\ut_{\bi\bj} + \l\ut_{\bi\bj}(x) ; & \\
  ((\bi,l),m) & \to &((\bi-e_l + e_{l'},l'),m) 
    &\mbox{at rate}\quad \bl\uh_{\bi ll'} + \l\uh_{\bi ll'}(x) ; & \\[0.5ex]
  ((\bi,l),m) & \to &((\bi,l),m+e_{l'}) 
    &\mbox{at rate}\quad \bl\uf_{\bi ll'} + \l\uf_{\bi ll'}(x) ; & \\[0.5ex]
  ((\bi,l),m) & \to & (\D,m) 
    &\mbox{at rate}\quad \bd'_{\bi l} + \d'_{\bi l}(x) ; & \\[0.5ex]
  ((\bi,l),m) & \to &(l,m) 
    &\mbox{at rate}\quad i_l^{-1}\{\bg_{\bi l} + \g_{\bi l}(x)\}; & \\[0.5ex]
  (l,m)& \to &((\bi+e_l,l),m)
    &\mbox{at rate}\quad x_\bi\s_{l\bi}(x); &  \\[0.5ex]
  (l,m)& \to &(\D,m) 
    &\mbox{at rate}\quad \bz_l + \z_l(x) . &  
\end{array}
\en
Here, the quantities $\bl\ui_{\bi ls}$ and $\l\ui_{\bi ls}(x)$ represent the rates at which, 
in a type~$\bi$ patch, an
animal of variety~$l$ produces offspring in the composition~$s$, and they would form a part
of the rates $\bl_{\bi,\bi+s}$ and $\l_{\bi,\bi+s}(x)$; they are assumed not to depend
on~$m$.  Similar considerations apply to the quantities $\bl\ut_{\bi\bj}$ amd~$\l\ut_{\bi\bj}(x)$, which
relate to events changing the composition of the patch containing~$\aa$ 
that do not result in offspring for~$\aa$ or a change in its variety, including migration of
other animals from the patch or the arrival of migrants.  Thus, for instance, one might have
$\bl_{\bi,\bi+e_l} = \f_{1l} i_l$, $\bl_{\bi,\bi-e_l} = \f_{2l} i_l$, $\bg_{\bi l} = i_l\f_{3l}$
and $\s_{l\bi}(x) = \s_{l\bi}$, $1\le l\le d$,
corresponding to constant {\it per capita\/} birth, death, migration and immigration rates $\f_{1l}$, $\f_{2l}$,
$\f_{3l}$ and~$\s_{l\bi}$ of individuals of 
variety~$l$.  These would imply $\bl\ui_{\bi le_l} = \f_{1l}$, $\bl\ut_{\bi,\bi+e_l} = (i_l-1)\f_{1l}$,
$\l\ut_{\bi,\bi+e_l}(x) = x_l\s_{l\bi}$,
and $\bl\ut_{\bi,\bi-e_l} = (i_l-1)(\f_{2l} + \f_{3l})$ for transitions only involving $l$-animals,
and, for $l' \neq l$, $\bl\ut_{\bi,\bi+e_{l'}} = i_{l'}\f_{1l'}$,
$\l\ut_{\bi,\bi+e_{l'}}(x) = x_{l'}\s_{l'\bi}$,
and $\bl\ut_{\bi,\bi-e_{l'}} = i_{l'}(\f_{2l'} + \f_{3l'})$.
The transition rates $\bl\uh_{\bi ll'}$ and $\l\uh_{\bi ll'}(x)$ relate to events that change $\aa$'s variety
from $l$ to~$l'$;
it is tacitly assumed that no other changes take place when this happens, but more general possibilities
could have been allowed.  The rates $\bl\uf_{\bi ll'}$ and $\l\uf_{\bi ll'}(x)$ relate to births
of migrants as offspring of an $l$-animal.
The rates $\bd'_{\bi l} \ge \bd_\bi$ and $\d'_{\bi l}(x) \ge \d_\bi(x)$ include a contribution from the 
mortality rate of an animal of variety~$l$ in a patch of type~$\bi$, in addition to 
the rate of destruction of the patch itself. 
As for the single patch dynamics,  the rates for the process~$Z_N$ at time~$t$ are obtained by
replacing $x$ with~$x^N(t)$ in the expressions~\Ref{individual-rates}.

These constructions immediately suggest approximating the processes $Y_N$ and~$Z_N$ by
random processes $Y$ and~$Z$, in which the transition rates at time~$t$ are obtained by replacing
$x$ by~$x(t)$ in \Ref{patch-rates} and~\Ref{individual-rates}.  
Consider first the processes $Y_N$ and~$Y$.  Suppose, for some $\d > 0$, that the functions 
$\l_{\bi\bj}$, $\g_{\bi l}$, $\s_{l\bi}$ and~$\d_\bi$ are all of uniformly bounded Lipschitz $\m$-norm, 
for~$x$ in a set $B_{T,\d} := \{x \in \xx'\colon\,\inf_{0\le t\le T}\|x-x(t)\|_\m \le \d\}$ of
points close to the deterministic trajectory $(x(t),\,0\le t\le T)$.  Then,
in view of~\Ref{patch-rates}, the jump rates of $Y_N$ and~$Y$ at any 
time $t\in[0,T]$ differ only by a small amount, on the event that 
$\sup_{0\le t\le T}\|x^N(t) - x(t)\|_\m \le \e_N$, provided that~$N$
is large enough that $\e_N \le \d$.  
Indeed, defining  $f^* := \sup_{x \in B_{T,\d}}|f(x)|$ for any $f\colon\,\xx\to\re$, and setting
\[
     |Df|(x)\ :=\ \limsup_{\e\to0}\sup_{0 < \|y-x\|_\mu < \e}\{|f(y) - f(x)| / \|y-x\|_\mu\},
\]
it follows that, if $|x-x(t)| \le \e < \d$ and $0\le t\le T$, then the sum of the differences of the 
transition rates out of $x$ and~$x(t)$ is bounded by
\eqs
  &&\sup_{\bi\in\zz_1} \Blb\sum_{\bj\in\zz_1}|\l_{\bi\bj}(x) - \l_{\bi\bj}(x(t))| + 
    \slid|\g_{\bi l}(x) - \g_{\bi l}(x(t))| \right.\\
   &&\qquad\left. \mbox{} +
     \slid|x_l\s_{l\bi}(x) - x_l(t)\s_{l\bi}(x(t))| + |\d_\bi(x) - \d_\bi(x(t))|\Brb \Le \e D_Y(T,\d),
\ens
where, writing $\hs_{l\bi}(x) := x_l\s_{l\bi}(x)$, we define
\[
    D_Y(T,\d) \Def
     \sup_{\bi\in\zz_1} \Blb \sum_{\bj\in\zz_1} |D\l_{\bi\bj}|^* + \slid \{|D\g_{\bi l}|^* +  |D\hs_{l\bi}|^*\}
         + |D\d_\bi|^* \Brb .
\]
Thus, until the time at which first $\|x^N(t) - x(t)\|_\m > \e_N$,  
the aggregate difference between the jump rates of the processes $Y_N$ and~$Y$ is bounded by
$\e_N D_Y(T,\d)$, if also $t \le T$.
 This immediately leads to the following theorem. 

\begin{theorem}\label{patch-approximation}
Suppose that \Ref{LLN-approx-2} holds, and that $D_Y(T,\d) < \infty$ for some $\d > 0$.  
Then, if $Y_N(0) = Y(0)$ and $\e_N \le \d$, 
the processes $Y_N$ and~$Y$ can be constructed on the same probability space in such a way that
\[
      \pr[Y_N(t) = Y(t)\ \mbox{for all}\ 0\le t\le T]\ \ge\ 1 - \{ T \e_N D_Y(T,\d) + P_T(N,\e_N)\} .
\]
\end{theorem}

\begin{proof}
Let $Y_1$ and~$Y_2$ be time-inhomogeneous Markov processes on a countable state space~$\yy$, 
with transition rates $q_1(t,y,y')$ and~$q_2(t,y,y')$ respectively.  Starting with $Y_1(0) = Y_2(0) = y_0$,
the processes can be coupled by representing them as the marginals of a joint process $((Y_1(t),Y_2(t)),\,t\ge0)$,
whose transition rates at points on the diagonal are given by
\eqs
     q(t,(y,y),(y',y')) &:=& \min\{q_1(t,y,y'),q_2(t,y,y')\};\\
    q(t,(y,y),(y,y')) &:=& \{q_2(t,y,y') - q_1(t,y,y')\}_+;\\
    q(t,(y,y),(y',y)) &:=& \{q_1(t,y,y') - q_2(t,y,y')\}_+,
\ens
and with the components evolving independently when off the diagonal.  Let $\t := \inf\{t\ge0\colon\, Y_1(t) \neq Y_2(t)\}$,
and let $E_t^\h$ denote the event $\{\hbox{$Q(s,Y_1(s)) \le \h$}\hfil\break \mbox{for all }0\le s\le t\}$, where
\[
    Q(t,y) \Def \sum_{y'\in\yy} |q_2(t,y,y') - q_1(t,y,y')|.
\]
Then the one-jump process $(I[\{\t \le t\} \cap E_t^\h],\,t\ge0)$ has compensator
\[
    A_t \Def \int_0^{t\wedge\t} Q(s,Y_1(s)) I[E_s^\h]\,ds \Le \h t.
\]
This implies that, for any $T > 0$,
\[
   \pr[\{\t \le T\} \cap E_T^\h] \Eq \ex\{I[\{\t \le T\} \cap E_T^\h]\} \Eq \ex A_T \Le \h T,
\]
from which it follows that $\pr[\t \le T] \le \h T + \pr[(E_T^\h)^c]$.  Thus this construction realizes $Y_1$ and~$Y_2$ on
the same probability space, in such a way that the two remain identical up to time~$T$ with probability at least
$1 - (\h T + \pr[(E_T^\h)^c])$.

Now, taking $Y_N$ for~$Y_1$ and $Y$ for~$Y_2$, and setting $\h = \e_N D_Y(T,\d)$, the theorem follows
from~\Ref{LLN-approx-2}.  
\end{proof}

\medskip
Since all the transitions in~\Ref{patch-rates} involve a single patch, the theorem
generalizes easily to any group of~$K$ patches.  The transition rates for the process
$(Y_N\uis,Y_N\uts,\ldots,Y_N^{[K]})$ at time~$t$ from a state $(\bi\ui,\ldots,\bi\uK)$ to one in which
$\bi\uk$ is replaced by $\bi\ukd$, with $\bi\ukd$ either of the form $\bi\uk + \bj$, $\bj\in\Z^d$, 
or~$\D$, are given by the formulae in~\Ref{patch-rates}
with $\bi\uk$ for~$\bi$, and with $x^N(t)$ for~$x$.  
The rates for a vector
of independent processes $Y^{[k]}$, $1\le k\le K$, each distributed as~$Y$, with 
$Y^{[k]}(0) = \bi\uk$, are the corresponding
rates with $x(t)$ for~$x$.  This leads to the following corollary.

\begin{corollary}\label{patch-group}
 Under the conditions of Theorem~\ref{patch-approximation}, 
\eqs
     \lefteqn{ \pr[(Y_N\uis(t),\ldots,Y_N^{[K]}(t)) = (Y\uis(t),\ldots,Y^{[K]}(t))\ \mbox{for all}\ 0\le t\le T]}\\ 
            && \ge\ 1 - \{ K T \e_N D_Y(T,\d) + P_T(N,\e_N)\} .\phantom{XXXXXXXXXXX}
\ens
\end{corollary}

\noindent
Thus the joint distribution of~$K_N$ patches is asymptotically close to that of~$K_N$ independently
evolving patches over any fixed interval $[0,T]$, as $N\to\infty$, if $K_N \e_N \to 0$, $P_T(N,\e_N)\to0$
and $D_Y(T,\d) < \infty$ for some $\d > 0$.

For the life history process of an animal, the argument
for a single individual is very similar.  We consider the differences in the transition
rates~\Ref{individual-rates} with arguments $x^N(t)$ and~$x(t)$; defining
\eqs
   D_Z(T,\d) &:=& 
     \max_{1\le l\le d}\biggl(\sup_{\bi\in\zz_1} \Bigl\{ \sum_{s\in\Z_+^d} |D\l\ui_{\bi ls}|^*  
         + \sjZi |D\l\ut_{\bi\bj}|^* + \sum_{l'=1}^d |D\l\uh_{\bi ll'}|^*   \\
     &&\qquad\qquad   \mbox{} + \sum_{l'=1}^d |D\l\uf_{\bi ll'}|^* + |D\d'_\bi|^*
         +  |D\g_{\bi l}|^*  +  |D\hs'_{l\bi}|^* \Bigr\}  + |D\z_l|^* \biggr),
\ens
where $\hs'_{l\bi}(x) := x_\bi\s_{l\bi}(x)$, this gives the following result.

\begin{theorem}\label{individual-approximation}
Suppose that \Ref{LLN-approx-2} holds, and that $D_Z(T,\d) < \infty$ for some $\d > 0$.  
Then, if $\e_N \le \d$ and $Z_N(0) = Z(0)$, 
the processes $Z_N$ and~$Z$ can be constructed on the same probability space in such a way that
\[
    \pr[Z_N(t) = Z(t)\ \mbox{for all}\ 0\le t\le T]\ \ge\ 1 - \{T \e_N D_Z(T,\d) + P_T(N,\e_N)\} .
\]
\end{theorem}

For the joint distribution of a group of~$K$ animals, asymptotic independence is not quite
as straightforward, since all but the fourth and the last transitions in~\Ref{individual-rates} simultaneously
change the state of any other animal in the same patch.  Hence it is necessary to begin
with all animals in different patches, and the simple coupling breaks down once two of them
are to be found in the same patch.  This can only occur when a migrant enters a patch that
already contains another of the~$K$ animals.  For a given animal of variety~$l$, 
an upper bound for the maximum rate at which it can enter such a patch
is $N^{-1}(K-1) \sup_\bi |\s_{l\bi}|^*$, because the $(K-1)$ other
animals of the group can be in at most $K-1$ distinct patches, and $\s_{l\bi}(x) \le |\s_{l\bi}|^*$; 
and there are~$K$ animals that could migrate into such a patch.  Hence the event that no two of 
the~$K$ animals are
in the same patch during the interval $[0,T]$ has probability bounded by $K^2 N^{-1} \s^+$,
where $\s^+ := \sup_{\bi\in\zz_1} \max_{1\le l\le d} |\s_{l\bi}|^*$. This leads to
the following corollary.

\begin{corollary}\label{individual-group}
Suppose that \Ref{LLN-approx-2} holds, and that $D_Z(T,\d) < \infty$ for some $\d > 0$. 
Then, if  $\e_N \le \d$ and the~$K$ individuals are initially all in distinct patches, we have
\eqs
     \lefteqn{ \pr[(Z_N\uis(t),\ldots,Z_N^{[K]}(t)) = (Z\uis(t),\ldots,Z^{[K]}(t))\ \mbox{for all}\ 0\le t\le T]}\\ 
            && \ge\ 1 - \{ K T \e_N D_Z(T,\d) + T K^2 N^{-1} \s^+ + P_T(N,\e_N)\} ,
\ens
where the $Z^{[k]}$, $1\le k\le K$, are independent copies of~$Z$ with $Z^{[k]}(0) = Z_N^{[k]}(0)$.
\end{corollary}

\noindent
Thus, if \Ref{LLN-approx-2} holds and $D_Z(T,\d) < \infty$ for some $\d > 0$, any group of~$K_N$ 
animals that are initially 
in different patches behaves asymptotically as a group of independent individuals, under 
the same asymptotic scenario as before, if also $N^{-1}K_N^2 \to 0$ as $N \to \infty$.

The model in Arrigoni~(2003) does not conform to our general prescription, because
migration is assumed to take place instantaneously, rather than by way of an intermediate
migration state.  However, the state dependent elements of its transition rates are locally
uniformly Lipschitz, and~\Ref{LLN-approx-2} holds, so that analogous theorems hold for this
model as well.  We do not include instantaneous migration in our general formulation, partly because
it seems unrealistic, but mainly because, for the methods in [BL]~(2012a) to be applied, 
only rather restrictive choices can be allowed for
the migration transitions.  For instance, in the Arrigoni model, it is important
that the migration rate $\bg_{i}$ out of patches with $i$ individuals is given by $\bg_{i} = \g i$; variants in which $i^{-1}\bg_i$ increases with~$i$ would not 
lead to a locally Lipschitz drift~$F$ in~\Ref{F-split} below.

\section{Establishing the law of large numbers}\label{LLN}
\setcounter{equation}{0}
We now need to prove that~\Ref{LLN-approx-2} holds.  For this, we need to find 
conditions on the transition rates in I--VI that allow us to
apply the results of [BL]~(2012a) to the process~$X^N$.  First,
we need to make some small modifications
to the setting in the previous section.  We start by augmenting the type space~$\zz$
to~$\tzz$,
by substituting $\tzz_2 := \{1,2,\ldots,d\} \times \{0,1\}$ for~$\zz_2$, where the
type~$(l,1)$ replaces the previous type~$l$ in representing an individual of variety~$l$
in migration, and type~$(l,0)$ is to be thought of as an unused place available for a
migrant of variety~$l$.  Then, in transitions IV and IV$'$, $e(l)$ is replaced by $e(l,1) - e(l,0)$ and, in
transitions V and~VI, $-e(l)$ is replaced by $e(l,0) - e(l,1)$ and $X_l$ by $X_{l1}$.  
The number $X_{l0}$ of patches of type $e(l,0)$ can be deduced from the number~$X_{l1}$ 
of $e(l,1)$ patches, since the sum $X_{l1} + X_{l0}$ remains constant in all transitions,
and is therefore always the same as its initial value.  However, to prevent the number
of type~$(l,0)$ patches becoming negative, the process~$X^N$ has to be stopped at the 
time~$\t_{0,N} := \inf\{t\ge0\colon\, \min_{1\le l\le d} X_{l0}^N = 0\}$.
So that this has little effect on the process, $X^N(0)$ is chosen with $X^N_{l0} \ge Nh_l$,
$1\le l\le d$, with the~$h_l$ so large that, for fixed~$T$, the event $\{\t_{0,N} \le T\}$ has 
asymptotically small probability as $N\to\infty$.
The reason for introducing the empty migration patches will emerge shortly.

\subsection{A priori bounds}\label{moments-sect}
We now introduce a measure~$\nu$ of the size of a patch, defining $\n(l,0) = \n(l,1) := 1$
for $1\le l\le d$, and $\n(\bi) := \|\bi\|_1+1$, one more than the number 
of individuals in a type~$\bi$ patch.  More flexible choices for~$\nu$ are allowed in [BL]~(2012a), but
this suffices here.  It is then necessary to make assumptions ensuring that,
for enough values of $r\in\Z_+$, the empirical moments $S_r(x^N(t)):= \sjtZ \n(z)^rx^N_z(t)$ 
remain bounded with high probability as~$N$ increases, if they are initially bounded. 
Let~$J$ denote a finite linear combination of coordinate 
vectors in~$\tzz$. 
Let~$\jj$ denote the jumps~$J$ that appear in the transitions I--VI, with the above modification
replacing~$e(l)$ by $e(l,1) - e(l,0)$, and let the associated
transition rates be denoted by $N\a_J(x)$.  Note that we can suppose that $x \in \xxs$,
if the~$l$ coordinates in~$\zz$ are identified with the $(l,1)$ coordinates in~$\tzz$,
since the values $x_{(l,0)}$ do not appear in the expressions for the transition rates I--VI.
For $J := \sum_{k=1}^K a_k e(\bj\uk) \in \jj$, write
\eq\label{nu(J)}
    \n_r^+(J) \Def \sum_{k=1}^K a_k \{\n(\bj\uk)\}^r,
\en
and, for $r\in\Z_+$, define
\eq\label{moments}
  U_r(x) \Def \sjj \a_J(x)\n_r^+(J); \quad V_r(x) \Def \sjj \a_J(x)\{\n_r^+(J)\}^2.  
\en
Then, in order to be able to apply the theorems of [BL]~(2012a), we assume that, 
for some $r\ui \ge 1$ and for all $0 \le r \le r\ui$,
\eq\label{assn1}
    \sjj \a_J(N^{-1}X)|\n_r^+(J)| \ <\ \infty \ \mbox{for each}\ X \in \xx,
\en
and that, for suitable constants $k_{rl}$ and all $x\in\xxs$,
\eq\label{assn2}
  \begin{aligned}
      U_0(x) &\Le k_{01}S_0(x) + k_{04}; \\
      U_1(x) &\Le k_{11}S_1(x) + k_{14}; \\
      U_r(x) &\Le \{k_{r1} + k_{r2}S_0(x)\}S_r(x) + k_{r4}, \qquad2\le r\le r\ui,
   \end{aligned} 
\en
and, for some $r\ut \ge 1$,
\eq\label{assn3}
  \begin{aligned}
      V_0(x) &\Le k_{03}S_1(x) + k_{05}; \\
      V_r(x) &\Le k_{r3}S_{p(r)}(x) + k_{r5},\qquad 1\le r \le r\ut,
   \end{aligned} 
\en
are satisfied, where $1 \le p(r) \le r\ui$ for $1\le r\le r\ut$.

In our setting, satisfying the condition~\Ref{assn1} is straightforward except
for the transitions of the form~II, since, for $X\in\xx$, only finitely many
of the~$X_\bi$ are non-zero; and transitions of the form~II are also the only ones that
make positive contributions to~$U_0(x)$.
One plausible assumption, covering these and later conditions, is to require that
\eq\label{beta-cond}
    \b_\bj(x) \le c'_\bj(\|x\|_1 + 1), \quad\mbox{where}\quad 
               \sjZi c'_\bj \{\n(\bj)\}^r < \infty\ \ \mbox{for each}\ r\in\Z_+.
\en
Here, and in what follows, $c$ and~$c'$ are used to denote generic constants.
If the types $(l,0)$ had not been introduced, there
would also be positive contributions of $\sjZi X_\bj\bg_{\bj l}$ to $U_0(x)$ from transitions~IV, and
the most natural assumption for the value of~$\bg_{\bj l}$ is $\g_l j_l$, for some
constant~$\g_l$, corresponding to a constant {\it per capita\/} migration rate for
$l$-individuals.
Thus $\sjZi X_\bj\bg_{\bj l}$ would be bounded by a multiple of $S_1(x)$,
rather than by a multiple of~$S_0(x)$, and so would not have come within the scope
of [BL]~(2012a).  For the remaining conditions concerning $U_r(x)$, $r\ge1$, it is
enough to assume that, for $\bi\in\zz_1$ and for all $x\in\xxs$,
\eqa
   \sjZi \bl_{\bi\bj} + 
   \sjZi (\bl_{\bi\bj} + \l_{\bi\bj}(x))\{\n(\bj) - \n(\bi)\}_+ &\le& c\n(\bi); \label{lambda-cond-1a} \\
   \sjZi (\bl_{\bi\bj} + \l_{\bi\bj}(x))(\{\n(\bj)\}^r - \{\n(\bi)\}^r)_+ &\le& c\{\n(\bi)\}^r(\|x\|_1+1), 
           \label{lambda-cond-1b}
\ena
and that, for $1\le l\le d$ and for all $x\in\xxs$,
\eq\label{sigma-cond}
    \s_{l\bi}(x) \Le c;\quad \sjZi x_\bj \s_{l\bj}(x) \Le c.
\en
For the conditions concerning $V_r(x)$, $r\ge0$, with $p(r) = 2r+1$ as in [BL]~(2012), we assume further that,
for $\bi\in\zz_1$ and $1\le l\le d$ and for all $x\in\xxs$,
\eq\label{delta-gamma-cond}
   \bd_\bi + \d_\bi(x) \Le c\n(\bi),\quad\bg_{\bj l} + \g_{\bi l}(x)  \Le c\n(\bi)
             \quad\mbox{and}\quad \bg'_{\bi l} + \g'_{\bi l}(x) \Le c\n(\bi), 
\en
and that
\eq\label{lambda-cond-2}
   \sjZi (\bl_{\bi\bj} + \l_{\bi\bj}(x))(\{\n(\bj)\}^r - \{\n(\bi)\}^r)^2 \Le c\{\n(\bi)\}^{2r+1}.
\en

\subsection{The deterministic equation}\label{DE-sect}
The process $x^N := N^{-1}X^N$ has infinitesimal drift~$F_0(x)$, $x\in\xxs$, whose components
are formally given by
\eqa
   F_{0;\bi}(x) &:=& \sjZi x_\bj\{\bl_{\bj\bi} + \l_{\bj\bi}(x)\} - x_\bi\sjZi\{\bl_{\bi\bj} + \l_{\bi\bj}(x)\}
       - x_\bi\{\bd_{\bi} + \d_{\bi}(x)\} \non\\
    &&\mbox{}\   + \b_\bi(x) + \slid x_{\bi+e_l}\{\bg_{\bi+e_l,l} + \g_{\bi+e_l,l}(x)\}
         - x_\bi \slid \{\bg_{\bi l} + \g_{\bi l}(x)\} \non\\
    &&\mbox{}\qquad + \slid x_{l1}\{x_{\bi-e_l}\s_{l,i-e_l}(x) - x_\bi\s_{l\bi}(x)\}, \label{F0-def-1}
\ena
\nin for $\bi \in \zz_1$, and, for $1\le l\le d$,
\eqa
  F_{0;l1}(x)  &:=& \sjZi x_\bj\{\bg_{\bj l} + \g_{\bj l}(x) + \bg'_{\bj l} + \g'_{\bj l}(x)\} \non\\
    &&\qquad\mbox{} - x_{l1} \sjZi x_\bj\s_{l\bj}(x)
      - x_{l1}\{\bz_l + \z_l(x)\};\label{Fhat-def-2}
\ena
these expressions only make sense if the $\bj$-sums are all finite.
The drift in the $(l,0)$ coordinate is given by $-F_{0;l1}(x)$, but we do not use
it explicitly.  Thus, for $x \in \xxs$ such that $F(x)$ exists, we can write
\eq\label{F-split}
   F_0(x) \Def Ax + F(x),
\en
to be interpreted as an element of $\re_+^{\zz}$, where
\eq
  \begin{aligned}
   A_{\bi\bj} \,&:=\, \bl_{\bj\bi} + \slid \bone_{\{\bj=\bi+e_l\}}\bg_{\bj l},\quad \bi\ne \bj \in\zz_1; \\[-1ex]
   A_{\bi\bi} \,&:=\, - \sjZi \bl_{\bi\bj} - \bd_\bi - \slid \bg_{\bi l},\quad \bi\in\zz_1;
            \\
   A_{\bi l} \,&:=\, 0, \  A_{l\bi} \,:=\, \bg_{\bi l} + \bg'_{\bi l}, \  A_{ll}  \,:=\, -\bz_l, 
        \ A_{ll'}  \,:=\, 0, \quad \bi\in\zz_1,\,1\le l,l'\le d,
   \end{aligned}\label{A-def-3}
\en
with~$l$ in the indices of~$A$ as shorthand for $(l,1)$;
and where
\eqa
  \lefteqn{F_{\bi}(x) \Def \sjZi x_\bj\l_{\bj\bi}(x) - x_\bi\sjZi \l_{\bi\bj}(x)
       + \b_\bi(x)  - x_\bi \d_{\bi}(x) + \slid x_{\bi+e_l} \g_{\bi+e_l,l}(x)} \non\\
  &&\mbox{}  \qquad\quad- x_\bi \slid \g_{\bi l}(x) 
           + \slid x_{l1}\{x_{\bi-e_l}\s_{l,\bi-e_l}(x) - x_\bi\s_{l\bi}(x)\},\phantom{XXXX}
                 \phantom{XX} \label{F-def-1}
\ena
for $\bi\in\zz_1$, and, for $1\le l\le d$,
\eqa
  F_{l1}(x) &:=& \sjZi x_\bj\{\g_{\bj l}(x) + \g'_{\bj l}(x)\} - x_{l1} \sjZi x_\bj\s_{l\bj}(x)
      - x_{l1} \z_l(x).  \phantom{XX}  \label{F-def-2}
\ena
The reason for splitting the drift as above is to treat models in which the
transition rates are not bounded as~$\n(\bi)$ increases --- migration, birth
and death rates proportional to the numbers of individuals in a patch are very 
natural --- enabling the theory of perturbed linear operators to be applied.  

We first assume that there is a real $\m\in[1,\infty)^\zz$ such that, for some $w\ge0$,
\eq\label{A-cond}
    A^T \m \Le w\m,
\en 
and use it to define the $\m$-norm 
\eq\label{norm-def}
   \|x\|_\m \Def \sjZ \m(z)|x_z| \quad\mbox{on}\quad \xxs_\m := \{x \in \re^\zz\colon\, \|x\|_\m < \infty\},
\en
with $x_l$ identified with~$x_{l1}$ as before.
Note that, if~\Ref{A-cond} is assumed, we must have $\sjZ \bl_{\bi z}\m(z) < \infty$ for each~$\bi$.
Then, as in [BL]~(2012a, Theorem~3.1), there exists a $\m$-strongly continuous semigroup $\{R(t),\,t\ge0\}$
with elementwise derivative $R'(0)=A$.
Furthermore, if~$F:\xxs_\m\to\xxs_\m$ is locally $\m$-Lipschitz and $\|x(0)\|_\m < \infty$, the integral equation 
\eq\label{deterministic}
    x(t) \Eq R(t)x(0) + \int_0^t R(t-u)F(x(u))\,du
\en
has a unique, $\m$-continuous solution on $[0,T]$ for any $0 < T < \tmax$, 
for some $\tmax \le \infty$.  This~$x$ is the deterministic
curve that approximates~$x^N(t)$ when $x^N(0)$ is $\m$-close enough to~$x(0)$. 

From now on, we take $\m(\bj) := \|\bj\|_1 + 1$ for $\bj\in\zz_1$ and $\m(l) := 1$ for $1\le l\le d$.
Inequality~\Ref{A-cond} is then satisfied if
\eq\label{A-mu-cond}
     \sjZi \bl_{\bi\bj}(\m(\bj) - \m(\bi)) +  \slid (\bg_{\bi-e_l,l} - \bg_{\bi l})\m(\bi-e_l)
      + \slid \bg'_{\bi l} \Le w\m(\bi)
\en
for all $\bi\in\zz_1$.  In order then to deduce that~$F\colon \xxs_\m \to \xxs_\m$  
is locally $\m$-Lipschitz, 
sufficient conditions are that, for $1\le l\le d$ and $\bi \in \zz_1$, and for any $R > 0$,
\eq
  \begin{aligned}
   &\s_{\bi l}(x), \d_\bi(x), \g_{\bi l}(x), \g'_{\bi l}(x), \z_l(x)\ \mbox{and}\ \sjZi \l_{\bi\bj}(x)\ 
         \mbox{are uniformly bounded, and}\\
   & \d_\bi(x), \g_{\bi l}(x), \g'_{\bi l}(x), \s_{\bi l}(x)\ \mbox{and}\ \z_l(x)\ \mbox{are $\m$-uniformly Lipschitz, in}\
      x \in B_R; \\[1ex]  
   &\sjZi |\l_{\bi\bj}(x) - \l_{\bi\bj}(y)| 
           \Le c\|x-y\|_\m, \quad \sjZi |\b_\bj(x)-\b_\bj(y)|\m(\bj) \Le c\|x-y\|_\m, \\              
   &\sjZi |\l_{\bi\bj}(x) - \l_{\bi\bj}(y)|\m(\bj) \Le c\m(\bi)\|x-y\|_\m 
              \ \ \mbox{and}\ \ \sjZi \l_{\bi\bj}(x)\m(\bj) \Le c\m(\bi),\\
   &\qquad  \mbox{uniformly in}\          x,y \in B_R, 
    \end{aligned}\label{Lip-conds} 
\en
for suitable constants $c = c_R$, where $B_R$ is the ball of radius~$R$ in~$\xxs_\m$.

\subsection{The law of large numbers approximation}\label{jumps-sect}
In order to apply the results of [BL]~(2012a), we still need to check that their
Assumption~4.2 is satisfied.  Part~(1) is satisfied with $r_\mu=1$, because
$\mu(z) = \nu(z)$ for all $z\in\tzz$. For Part~(2), we define 
$\z(\bi) := (\|\bi\|_1+1)^{2d+5}$ for $\bi\in\zz_1$ and $\z(l,1) := \z(l,0) = 1$
for $1\le l\le d$, and observe that then, using
conditions \Ref{lambda-cond-1a} and~\Ref{delta-gamma-cond}, the sum
\[
     Z \Def \sjZi \frac{\mu(\bj)(A_{\bj\bj}+1)}{\sqrt{\z(\bj)}} 
       \Eq O\Bigl(\sum_{j\ge0} j^{(d-1) + 2 - (d+5/2)}\Bigr) \ <\ \infty.
\]
This implies that [BL]~(2012a, Assumption~4.2(2)) is satisfied, provided that~$\z$
satisfies [BL]~(2012a, Assumption~(2.25)).  Defining
$f(J) := \sum_{k=1}^K |a_k| \z(\bj\uk)$ when $J :=  \sum_{k=1}^K a_k \bj\uk$,
this in turn requires that
\eq\label{assn-4.2}
     \sjj \a_J(x) f(J) \Le \{k_1 S_r(x) + k_2\},\quad x\in\xxs,
\en
for some constants $k_1$ and~$k_2$ and for some $r \le r\ut$.
However, this also follows from conditions \Ref{lambda-cond-1a}--\Ref{delta-gamma-cond},
if $r = 2d+6$.  Hence it is necessary to have $r\ut \ge 2d+6$ in~\Ref{assn3}
and thus $r\ui \ge 4d+13$ in~\Ref{assn2}.

Suppose now that the assumptions  \Ref{beta-cond}--\Ref{lambda-cond-2} of Section~\ref{moments-sect},
and \Ref{A-cond}, \Ref{A-mu-cond} and~\Ref{Lip-conds} of Section~\ref{DE-sect}, 
are all satisfied.  Then
it follows from [BL]~(2012a, Theorem~4.7) that, for a sequence of initial conditions satisfying
\eq\label{IC-1}
  x_N(0) \in \xxs, \ N\ge1;\quad  S_{2d+6}(x_N(0)) \le C_*\ \mbox{for some}\ C_* < \infty, 
\en
and
\eq\label{IC-2}
   \|x_N(0) - x(0)\|_\m \Eq O(N^{-1/2}\sqrt{\log N})\quad \mbox{for some}\ x(0) \in \xxs_\m,
\en
the deterministic approximation~\Ref{LLN-approx-2} holds for any~$T$, with
\[
   \e_N \Eq k_T N^{-1/2}\sqrt{\log N}\quad\mbox{and}\quad P_T(N,\e_N) \Eq k_T' N^{-1}\log N,
\]
for suitably chosen constants $k_T$ and~$k_T'$.  Note that
equation~\Ref{deterministic} remains the same, whatever the values~$h_l$, $1\le l\le d$, chosen
as lower bounds for~$x^N_{l0}$.  Hence, in view of this approximation, it follows that
the event $\{\t_{0,N} \le T\}$ has probability at most $P_T(N,\e_N)$ if the~$h_l$ are chosen to satisfy 
$h_l \ge \sup_{0\le t\le T}x_{l1} + \d$ for each~$l$, for some $\d>0$, whenever~$N$ is so 
large that $\e_N < \d$.
Thus, under the above conditions on the rates for the transitions I--VI,
the results of Section~\ref{main} all hold, with the above values of $\e_N$ and $P_T(N,\e_N)$.
In particular, groups of patches or of animals of sizes $K_N = O(N^\a)$, for any $\a < 1/2$, behave
asymptotically independently.

\begin{remark}
{\rm The assumptions concerning the transition rates are rather general, and cover many
biologically useful models.  They can be extended somewhat, as far as the permissible
variation with~$x$ is concerned, by noting that the inequality~\Ref{assn3}, for $r\ge1$,
could be replaced by 
$$
    V_r(x) \le k_{r3}S_{p(r)}(x)(1+S_0(x)) + k_{r5}; 
$$
this would require only minor modification to the proof of [BL]~(2012a, Theorem~2.4).  
For our purposes, the bounds in \Ref{delta-gamma-cond} and~\Ref{lambda-cond-2} could 
then be relaxed by multiplying their right hand sides by a factor $(\|x\|_1 + 1)$. 
However, it is not obvious that the inequality
in~\Ref{lambda-cond-1a} can be relaxed in this way, and this restricts the freedom
for~$\l_{\bi\bj}(x)$ to vary with~$x$.}
\end{remark}

\section{Examples}\label{examples}
\setcounter{equation}{0}
\subsection{Example 1: The finite patch size models of Metz \& Gyllenberg~(2001)}  
The first model, with~$N$ patches and just one variety of animal, has transitions of the form I--VI,
with index set $\Z_+ \cup \{D\}$, where~$D$ is used here as index for the migrants
(Metz \& Gyllenberg use~$D$ to denote our~$x_D$). In their notation, in a patch with~$i$ occupants, 
the birth rate is $\bl_{i,i+1} := i\l_i(1-d_i)$, the death rate $\bl_{i,i-1} := i\m_i$,
the catastrophe rate $\bl_{i,0} := \g_i$ and the birth rate of (juvenile) migrants $\bg'_{iD} := i\l_id_i$; 
here, $0\le d_i\le 1$ for all~$i$.
The arrival rate of a migrant into an $i$-patch is~$\s_{Di}(x) := \a s_i$, where $0 \le s_i \le 1$ for all $i$, and the death rate
of a migrant is~$\z_D := \m_D$.
All other transition rates are zero; in particular, there is none of the explicit 
dependence on~$x$ that would be allowed in our formulation, for functions such 
as~$\l_{ij}(x)$.  

We take $\n(i) = \m(i) = i+1$, $i\in\Z_+$, and $\n(D) = \m(D) = 1$.  Then
assumption~\Ref{beta-cond} is trivially satisfied, and \Ref{lambda-cond-1a} and~\Ref{lambda-cond-1b}
require $\l_i$ to be bounded (so, as is reasonable, the 
{\it per capita\/} birth rate of an animal is to be bounded), in which case~\Ref{delta-gamma-cond}
is also satisfied.
For~\Ref{sigma-cond}, we require $s_i$ to be bounded, which is satisfied since $s_i$ are assumed to be probabilities.
Condition~\Ref{lambda-cond-2} also involves $\g_i$ and~$\m_i$, and is satisfied if, in
addition, $\m_i$ and $i^{-1}\g_i$ are bounded in $i\ge1$.  The conditions~\Ref{Lip-conds} are 
trivially satisfied, and~\Ref{A-mu-cond} is satisfied for
\[
   w \Def \sup_{i\ge1}\{\l_i(1-d_i) - \m_i - i^{-1}\g_i + ((i-1)\l_{i-1}d_{i-1} - i\l_id_i)\},
\]
finite if also $u_i := (i-1)\l_{i-1}d_{i-1} - i\l_id_i$ is bounded above in $i\ge1$.
The quantity~$u_i$ is the amount by which the total migration from a patch declines, when
the number of individuals in the patch increases from $i-1$ to~$i$, and for this to be
bounded is again an entirely reasonable hypothesis.  Finally, the quantities $D_Y(T,\d)$
and~$D_Z(T,\d)$ are bounded, since the~$s_i$ are bounded.
Hence, assuming that
\eq\label{MG-assn-1}
      \l_i, \m_i, i^{-1}\g_i, \ \mbox{and}~u_i\ \mbox{are bounded}, 
\en
our theorems apply to the initial model of Metz \& Gyllenberg~(2001), for initial
conditions~$x^N(0)$ satisfying \Ref{IC-1} and~\Ref{IC-2}.  
As it happens, the authors restricted their
model by imposing a maximal number of animals per patch `to make life easy', so that~\Ref{MG-assn-1} 
is trivially satisfied in their context; but such a restriction is unnatural, and we have shown that it
can be replaced by~\Ref{MG-assn-1}.   Metz \& Gyllenberg use the deterministic 
approximation~$x := \{x(t),\,t\ge0\}$ as the basis for
their analysis, and this is justified over any fixed finite time 
interval $[0,T]$ by the discussion in Section~\ref{LLN}, provided that~$N$
is large enough.  

The results of Section~\ref{main} now show, in addition, that small groups of individuals
behave almost independently of each other, according to time inhomogeneous Markov
jump processes whose transition rates are determined by~$x$.  For a chosen patch~$\pp$, the
Markov process has transition rates at time~$t$ given by
\eq\label{patch-rates-MG}
\begin{array}{rlllllr}
  i& \to &i+1 
    &\mbox{at rate}\quad i\l_i(1-d_i) + x_D(t)\a s_i, &\quad i \ge 0; & \\[0.5ex]
  i& \to &i - 1 
    &\mbox{at rate}\quad i\m_i, &\quad i \ge 2; & \\[0.5ex]
  i& \to &0 
    &\mbox{at rate}\quad \g_i + \m_1 \bone_{\{1\}}(i), &\quad i \ge 1.  
\end{array}
\en
Any particular animal~$\aa$ is born either as a migrant, or in a patch.  
Once in a patch, it never migrates again.  Its Markov process has transition rates at time~$t$ given by
\eq\label{individual-rates-MG}
\begin{array}{rlllllr}
  (i,m) & \to &(i+1,m+1) 
    &\mbox{at rate}\quad \l_i(1-d_i) ; &i\ge1 \\[0.5ex]
  (i,m) & \to &(i+1,m) 
    &\mbox{at rate}\quad (i-1)\l_i(1-d_i) + x_D(t)\a s_i ; &i\ge2 \\[0.5ex]
  (i,m) & \to &(i-1,m) 
    &\mbox{at rate}\quad (i-1)\m_i ; &i\ge2 \\[0.5ex]
  (i,m) & \to &(i,m+1) 
    &\mbox{at rate}\quad \l_i d_i ; &i\ge1 \\[0.5ex]
  (i,m) & \to &(\D,m) 
    &\mbox{at rate}\quad \m_i + \g_i ; &i\ge1 \\[0.5ex]
  (D,0)& \to &(i,0)
    &\mbox{at rate}\quad \a x_{i-1}(t)s_{i-1}; &i\ge1  \\[0.5ex]
  (D,0)& \to &(\D,0) 
    &\mbox{at rate}\quad \m_D . &  
\end{array}
\en
In either case, the process depends on~$x(t)$ only through the arrival rates of
migrants into patches.

The second model of Metz \& Gyllenberg~(2001) has animals of two different varieties,
that interact through living in common patches, in that their {\it per capita\/} birth and death
rates $\l$ and~$\m$ and their migration parameters $d$ and~$s$ vary with the
entire composition $(i_1,i_2)$ of the populations of the two varieties in a patch.
Under assumptions analogous to~\Ref{MG-assn-1}, the deterministic process~$\{x(t),\,t\ge0\}$
with index set $Z_+^2 \cup \{D_1,D_2\}$
again acts as a good approximation to the random process~$x^N$, and small groups of individuals
and patches behave asymptotically almost independently. 
Sufficient conditions for this are bounded {\it per capita\/} birth, death, catastrophe and 
migrant arrival rates, together with $u_{i_1,i_2}$ being bounded in $i_1,i_2\ge0$, where
\[
    u_{i,j} \Def (i-1)\l_{i-1,j}d_{i-1,j} - i\l_{ij}d_{ij} 
                     + (j-1)\l_{i,j-1}^*d_{i,j-1}^* - j\l_{ij}^*d_{ij}^*;
\]
here, the starred quantities are those for the second variety, and the unstarred those for the
first. 

However,
Metz \& Gyllenberg are interested in using the approximation when just a small
number of animals of the second variety have been introduced into a resident
metapopulation consisting only of the first variety.  Under such circumstances,
the development of the introduced variety has an essentially random component
--- it may die out by chance, even if at a theoretical advantage --- making it 
more reasonable to treat it as a small group of individuals, of a different
variety, evolving at random among a resident population.  The following discussion
represents a theoretical justification for the analysis in  Metz \& Gyllenberg (2001, Section~2(d)).

We begin by choosing $x^N(0) = \tx^N(0) + N^{-1}K_Ne_{D_2}$, where~$\tx^N(0)$ is
an initial composition consisting only of individuals of the first variety,
and $\|\tx^N(0) - \tx(0)\|_\m = O(N^{-1/2}\sqrt{\log N})$ for some fixed
$\tx(0) \in \xxs_\m$, which thus also consists only of $1$-individuals.  
Then, in the transition rates for any Markov process approximating 
individual dynamics, the argument~$x(t)$ can be taken to be~$\tx(t)$, where~$\tx$ 
denotes the solution of~\Ref{deterministic} starting at~$\tx(0)$,  provided that 
$K_N = O(N^{\b})$ for any $\b < 1/2$, because
then $\|x^N(0) - \tx(0)\|_\m = O(N^{-1/2}\sqrt{\log N})$ also.  But since~$\tx(0)$
consists only of $1$-individuals, so does~$\tx(t)$ for all $t > 0$,
and~$\tx(t)$ is the solution to the deterministic equation for the initial
model of Metz \& Gyllenberg~(2001), with the parameters of the resident population.

Since a $2$-juvenile, once arrived in a patch, never leaves it,
the development of the introduced species is best described in terms of the
evolution of the patches that $2$-juveniles reach.  Each such patch
can be treated as an `individual',  and the $2$-migrants
that leave it as its offspring, up to the time at which the patch contains no more
$2$-individuals.  This patch process, of a `$p$-individual', can thus be interpreted as a 
life history process~$W$,
beginning with the juvenile $2$-migrant, whose offspring are the $2$-migrants that
leave its chosen patch.  The entire process begins with a group of~$K_N$ juvenile $2$-migrants,
and the $2$-migrant offspring of the resulting $p$-individuals in turn initiate new $W$-processes, 
so that the entire process, if the bound deduced from Corollary~\ref{individual-group}
is small, can be approximated by
a Crump--Mode--Jagers (CMJ) branching process (Crump \& Mode (1968a,b),
Jagers~(1968);  see also Jagers (1975, Chapter~6)).

Let $W(t) = ((i,j),m)$ indicate that, at time~$t$, the patch contains~$i$ $1$-individuals
and~$j$ $2$-individuals, and that~$m$ $2$-migrants have left the patch
up to time~$t$; if $(i,j)$ is replaced by~$\D$, this indicates that the initial juvenile 
and all of its offspring that did not migrate, if there were any, have died, and~$D_2$ is
used when the state consists of the single juvenile $2$-migrant, before it reaches a patch.
The transition rates of~$W$ at time~$t$ can then be expressed as 
\eq\label{individual-rates-MG-2}
\begin{array}{rlllllr}
  ((i,j),m) & \to &((i,j+1),m) 
    &\mbox{at rate}\quad j\l_{ij}^*(1-d_{ij}^*) ; &i\ge0,\,j\ge1 \\[0.5ex]
  ((i,j),m) & \to &((i+1,j),m) \\
    &&\qquad\hfill\mbox{at rate}&i\l_{ij}(1-d_{ij}) + \tx_D(t)\a s_{ij} ; &i\ge1,\, j\ge1 \\[0.5ex]
  ((i,j),m) & \to &((i-1,j),m) 
    &\mbox{at rate}\quad  i\m_{ij} ; &i\ge1,\,j\ge1 \\[0.5ex]
  ((i,j),m) & \to &((i,j),m+1) 
    &\mbox{at rate}\quad j\l_{ij}^* d_{ij}^* ; &i \ge 0,\,j\ge1 \\[0.5ex]
  ((i,j),m) & \to &((i,j-1),m) 
    &\mbox{at rate}\quad  j\m_{ij}^* ; &i\ge0,\,j\ge2 \\[0.5ex]
  ((i,j),m) & \to &(\D,m) 
    &\mbox{at rate}\quad \m_1^*\bone_{\{1\}}(j) + \g_{ij} ; &i\ge0,\, j\ge1 \\[0.5ex]
  (D_2,0)& \to &((i,1),0)
    &\mbox{at rate}\quad \a \tx_{i}(t)s_{i0}^*; &i\ge0  \\[0.5ex]
  (D_2,0)& \to &(\D,0) 
    &\mbox{at rate}\quad \m_D^* . &  
\end{array}
\en  
In particular, if the resident population started at an equilibrium of the deterministic
equations, so that $\tx(t) = \tx(0)$ for all~$t$, then these transition rates are time homogeneous.  
Note also that,
since the {\it per capita\/} birth rate of the second variety is uniformly bounded
over all patch compositions, comparison with a linear pure birth process shows that
the expectation of the square of the number of $2$-individuals that 
were ever alive during $[0,T]$ is bounded by $c_T K_N^2$, for a suitable $c_T < \infty$.
Hence the probability that any $2$-migrant, whenever it was
born, arrives during $[0,T]$ in a patch which
has already been visited by individuals of the second variety is of order $O(N^{-1}K_N^2)$,
and this is asymptotically small if $K_N = O(N^{\b})$ for any $\b < 1/2$.

Thus, in view of Corollary~\ref{individual-group}, the evolution of the introduced species
over any finite time interval $[0,T]$,
measured in terms of the number of juvenile migrants, is the same as that of a CMJ--branching 
process, with probability of order $O(N^{-1+2\b})$.
The individual life history consists of a period of migration, followed either
by death (with probability $\m_D^*/S$, where $S := \m_D^* + \sio \a\tx_i(0) s_{i0}^*$) or arrival 
in a patch (of type $(i,0)$
with probability $\a \tx_i(0) s_{i0}^*/S$), after which its subsequent life history  
follows that of the Markov process with rates~\Ref{individual-rates-MG-2}, started in the
state $((i,1),0)$.  In particular, each transition of this process in which the third component 
increases corresponds to the birth of a new juvenile migrant.  If $P(i,j,t)$ denotes the probability
$\pr[(W_1(t),W_2(t)) = (i,j) \giv W(0) = (D_2,0)]$, then the mean intensity of the offspring
process is $m(t) := \sio \sji P(i,j,t) j\l_{ij}^* d_{ij}^* \,dt$, and the mean number of offspring
is $\bm := \int_0^\infty m(t)\,dt \le \infty$.  

The approximation using a branching process gives a lot of insight into the development of
the introduced species.  In particular, if the equation $\int_0^\infty e^{-\rho t} m(t)\,dt = 1$
has a solution~$\rho > 0$ (which has to be the case if $1 < \bm < \infty$), then the
introduced species, if it becomes established, grows exponentially with rate~$\rho$, and the
probability that it becomes established from an initial population of~$K$ juvenile migrants
is $1 - q^K$, where~$q$ is the extinction probability of the Galton--Watson process,
starting with a single individual, whose
offspring distribution is the distribution of the total number of offspring in the CMJ--process. 
If $\bm \le 1$, the introduced species dies out with probability one.
However, the current theorems only guarantee this approximation to be valid over a fixed time 
interval $[0,T]$, and then for~$N$ sufficiently large.  In Barbour, Hamza, Kaspi \& Klebaner~(2013),
the development of an introduced species, including the branching approximation, is considered
over much longer time intervals, but
in the context of {\it finite dimensional\/} Markov population processes. It would be
interesting to establish analogous results in the current context.

Metz \& Gyllenberg~(2001) made the (intuitively obvious) conjecture that, if the introduced 
species has exactly the same parameters as the original, and is introduced in equilibrium, 
then $\bm = 1$.  This is equivalent to saying that, in equilibrium, each migrant generates
a process that results in an average of exactly one new migrant.  They were, however, unable 
to give a proof of this.  If the random process for finite~$N$ were ergodic, it would be
natural to use arguments based on long term time averages as the basis of a proof.
However, the finite~$N$ process
is eventually absorbed in the zero population extinction state, so such arguments
cannot be used.  However, we sketch a proof of the
conjecture, under assumptions that include those of Metz \& Gyllenberg, in the appendix.

\subsection{Example 2: Kretzschmar's (1993) model}
In Kretzschmar's (1993) model of parasitic infection, $N$ denotes the initial number
of hosts, these playing the role of patches.  The index $i \in \Z_+$ denotes the number of
parasites living in the host.  The model has transitions of the form I--VI, with
$\l_{i,i-1} := i\m$, $\l_{i,i+1} := \l \f(x)$, $\b_0(x) := \b\sio x_i \th^i$ and
$\d_i := \k + i\a$, all other transition rates being zero; here, $0 \le \th \le 1$,
and $\f(x) := \sji jx_j/(c + \|x\|_1)$ for some $c > 0$.  It is shown in [BL]~(2012a,
Example~5.1) that, if the initial conditions satisfy \Ref{IC-1} and~\Ref{IC-2}, then 
the law of large numbers approximation~\Ref{LLN-approx-2} holds
with  $\e_N = k_T N^{-1/2}\sqrt{\log N}$ and $P_T(N,\e_N) = k_T' N^{-1}\log N$, for 
suitably chosen constants $k_T$, $k_T'$, where, as usual, $\m(i) = i+1$. It is
also easy to check that $D_Y(T,\d) < \infty$ for all $T$ and~$\d$.  The patch
process~$Y$ on $\Z_+\cup\D$ has transition rates at time~$t$ given by
\eq\label{patch-rates-K}
\begin{array}{rlllllr}
  i& \to &i+1 
    &\mbox{at rate}\quad  \l \f(x(t)), &\quad i\ge0; & \\[0.5ex]
  i& \to &i - 1 
    &\mbox{at rate}\quad i\m, &\quad i \ge1; & \\[0.5ex]
  i& \to &\D 
    &\mbox{at rate}\quad \k + i\a,  &\quad i \ge0.   
\end{array}
\en
One way of looking at this process is as a superposition of Poisson processes. 
Each parasite on arrival decides independently either to die or to kill the host,
with probabilities $\m/(\m+\a)$ and $\a/(\m+\a)$ respectively.  The time of
this event is exponentially distributed with mean $1/(\m+\a)$.  Independently,
the host is killed after an exponentially distributed time with mean~$1/\k$.
Because of the independence of marked Poisson streams, given that the host is
alive at time~$T$, the number of
parasites living in it has a Poisson distribution with mean
\[
    \int_0^T \l\f(x(t)) e^{-(\m+\a)(T-t)}\,dt.
\]
Thus a cohort consisting of $K_N$ hosts of given age~$T$ would exhibit an
approximately Poisson distribution of parasites per host, if $K_N = O(N^{\g})$ for some
$\g < 1/2$.  Thus, within age classes, Kretzschmar's model does not generate
over-dispersed distributions of parasites per host, though mixing over age
classes in a sample may be expected to do so.   Even then, if $\a$ and~$\k$ are much 
smaller than~$\m$, and~$x$ is in equilibrium, the departure from Poisson may not be 
very noticeable, unless there are many young hosts (with ages comparable to~$1/\m$) 
in the sample.

\section*{Appendix}
\setcounter{equation}{0}
In this section, we establish the conjecture of Metz \& Gyllenberg~(2001) discussed
above.  For this purpose, we can take their single type model, since all individuals
behave in the same way.
Let~$Z$ denote the CMJ-branching process associated with the process~$W$ of Example~1,
when the underlying process~$x$ is in equilibrium.  Suppose first that its mean~$\bm$
exceeds~$1$, so that its extinction probability~$q$ is less than~$1$.  In this case,
given any $M>0$, there exists a finite time~$T_M$ such that
\[
    \pr_1[Z(T_M) > M] \ >\ (1-q)/2,
\]
where $\pr_1$ denotes probability starting from a single migrant.  Starting the $x_N$-process
close to the equilibrium~$\bar x$, there are $d_N \approx N\bar x_D$ migrants at time~$0$.
We assume that $\bar x_D > 0$, which is true, for instance, under the irreducibility
condition introduced below.
Let~$Z_N^j$ denote the process of migrant descendants of the $j$-th of them.  As noted
above, it has distribution close to that of~$Z$ for large~$N$, by Theorem~\ref{individual-approximation}.  Set
$I_j := I[Z_N^j(T_M) > M]$, and let~$N$ be so large that $\ex I_j =: p_N > (1-q)/2$.  
Then, because any two of the processes $Z_N^j$ and~$Z_N^k$, $k\neq j$, are asymptotically
independent as $N\to\infty$, by Corollary~\ref{individual-group}, it follows that $\ex(I_jI_k)= p_N^2 + o(1)$
as $N\to\infty$, implying in turn that $S_N := \sum_{j=1}^{d_N} I_j$ has $\ex S_N > N\bar x_D (1-q)/2$
for all~$N$ large enough, and that $\var S_N = o(N^2)$.  Thus, by Chebyshev's inequality,
$\pr[MS_N \ge MN\bar x_D(1-q)/4] \to 1$ as $N\to\infty$.  But this contradicts~\Ref{LLN-approx-2}
if~$M$ is chosen such that $M(1-q)/4 > 1$,
because $MS_N \le N x^N_D(T_M)$, and~\Ref{LLN-approx-2} implies that 
$\pr[Nx^N_D(T_M) \le N\bar x_D(1+\e)] \to 1$ for any $\e > 0$.

The proof of contradiction if $\bm < 1$ is more involved.  Recall that~$m(\cdot)$ denotes
the mean offspring measure of the CMJ-branching process~$W =: W_0$ starting with $W_0(0) = (D_2,0)$. Let
$m_i(\cdot)$ denote the mean offspring measure for the initial individual in the process~$W_i$, 
starting with $W_i(0) = ((i,1),0)$.  All of its migrant children have mean offspring measure~$m$,
but the initial individual in general has a different measure.
Let $n_i(t)$ denote the mean number of migrants alive at time~$t$ in the process~$W_i$.
Then $n_i(t) = \int_0^t m_i(dv) n_0(t-v)$.  The assumption $\bm < 1$ implies that $\lti n_0(t) = 0$,
and so $\lti \sio i\bar x_i n_i(t) = 0$ also, if $\sio i\bar x_i \int_0^\infty m_i(dv) < \infty$,
by dominated convergence.  The latter is true, if $\sio i\bar x_i < \infty$ and if 
$\sup_i \int_0^\infty m_i(dv) = m^* < \infty$.

We now make four assumptions.  The first three are that $0 < \sio i\bar x_i < \infty$, 
that $\l^* := \sup_i \l_i < \infty$, and
that, for some $\e > 0$, there exists~$i_0$ such that $\m_i + \g_i \ge \l_i\{1 - (1-\e)d_i\}$ 
for all $i\ge i_0$.  
The fourth is an irreducibility assumption: we require that the birth, death and
catastrophe rates are such that a patch with~$i \ge 1$
occupants can evolve into a patch with~$i' \ge 0$ occupants, for $ i \neq i'  \le i^*$,
where~$i^*$ is the maximum possible number of occupants of a patch (infinity, if there is
no maximum); that $s_0 > 0$; and that $\l_i d_i > 0$ for some~$i \ge 1$.

The second of the assumptions ensures that mean proportion of the contribution to
$\ex X^N_D(t)$ arising from individuals in~$X^N(0)$ whose family trees do {\it not\/} remain coupled
to the corresponding branching process over any fixed interval $[0,T]$ is asymptotically negligible
as $N\to\infty$, for~$T$ fixed: the worst contribution from any such individual is $\exp\{\l^*T\}$,
and the proportion of them is asymptotically negligible as $N\to\infty$, 
by Theorem~\ref{individual-approximation}.  
The fourth assumption, together with~$\bm<1$, ensures that $\int_0^\infty m_i(dv) < \infty$ for each~$i$, 
since there is
then a positive probability that a migrant is at some time in a patch with $i-1$ other occupants,
and its total mean number of migrant offspring is finite.
The third assumption 
ensures that~$m^* < \infty$. This can be proved by analyzing a system of recurrence equations
satisfied by the quantities~$\int_0^\infty m_i(dv)$, showing that, in $i\ge i_0$,
$\int_0^\infty m_i(dv)$ is uniformly bounded by a quantity of the form $c_1 + c_2\int_0^\infty m_{i_0}(dv)$.  
This, combined with the first assumption,
shows that the contribution to $N^{-1}\ex X^N_D(t)$ arising from individuals for which the coupling is
maintained over $[0,T]$ is asymptotically close to $\sio i\bar x_i n_i(T)$ as $N\to\infty$, 
which can be made as small as desired by choosing~$T$ large enough.  Furthermore, because $\l^* < \infty$,
the variance of the contribution to $X^N_D(T)$ from any individual is uniformly bounded in~$i$, and
the correlation between the contributions from pairs of different individuals is asymptotically
small in~$N$, by Corollary~\ref{individual-group}.  Hence, with ever higher probability as 
$N\to\infty$, $N^{-1}X^N_D(T)$ stays close
to its (small) expectation.  However, for~$x_N$ in equilibrium, it has also to be asymptotically close 
to the fixed value $\bar x_D$, by~\Ref{LLN-approx-2}, and this is a contradiction, if $\bar x_D > 0$;
and this is the case, because of the fourth assumption.  

Metz \& Gyllenberg~(2001) actually assume that there is a largest index $i^* < \infty$.
In this case the conditions are typically satisfied, if~$i_0$ is taken equal to the largest index~$i^*$
in the third assumption.
However, there are some trivial possibilities where their conjecture is not true.  For instance,
if $i^*=1$ and $\mu_1 + \g_1 = 0$ and $\l_1 > 0$ (in which case, from the definition of~$i^*$,
$d_1 = 1$, and also $\s_1=0$), and if $\l_0 > 0$, one would have $\bar x_0 = 0$, $\bar x_1=1$ and $\bm = 0$,
but $\bar x_D = \l_1/\mu_D > 0$.  Of course, this is a biologically implausible scenario, and 
it violates both the third and fourth assumptions.

\section*{Acknowledgement}
ADB thanks the Department of Statistics and Applied Probability at the National University of Singapore,
and the mathematics departments of the University of Melbourne and Monash University, for
their kind hospitality while much of the work was undertaken. MJL is grateful to the Department of Statistics and Applied Probability at the National University of Singapore,
and the Department of Mathematics at the University of Melbourne for their hospitality.


\begin{thebibliography}{99}

 
\bibitem{a03} 
{\sc F.\ Arrigoni}  (2003) 
Deterministic approximation of a stochastic meta\-population model. 
{\em Adv.\ Appl.\ Probab.} {\bf 35}, 691--720. 
MR1990610

\bibitem{BHKK}
{\sc A.~D.~Barbour, K.~Hamza, H.~Kaspi \& F.~C.~Klebaner} (2013)
Escape from the boundary in Markov population processes.
{\tt arXiv:1312.5788}

\bibitem{ka93} 
{\sc A.~D.~Barbour \& M.~Kafetzaki} (1993) 
A host--parasite model yielding heterogeneous parasite loads. 
{\it J.\ Math.\ Biol.\/} {\bf 31}, 157--176.
 
\bibitem{Ba75}
{\sc A.\ D.\ Barbour \& M.\ J.\ Luczak} (2008)
Laws of large numbers for epidemic models with countably many types.
{\it Ann.\ Appl.\ Probab.\/} {\bf 18}, 2208--2238. 
MR2473655

\bibitem{BL10a}
{\sc A.\ D.\ Barbour \& M.\ J.\ Luczak} (2012a) 
A law of large numbers approximation for Markov population processes
with countably many types.  
{\it Prob.\ Theory Rel.\ Fields\/}~{\bf 153}, 727--757.
MR2948691

\bibitem{BL10b}
{\sc A.\ D.\ Barbour \& M.\ J.\ Luczak} (2012b)
Central limit approximations for Markov population processes with countably many types.
{\it Electr.\ J.\ Probab.\/} {\bf 17}, no.\ 90, 16pp.

\bibitem{CM68a}
{\sc K.\ S.\ Crump \& C.\ J.\ Mode} (1968a)
A general age-dependent branching process, I.
{\it J.\ Math.\ Anal.\ Appl.\/} {\bf 24}, 494--508. 

\bibitem{CM68b}
{\sc K.\ S.\ Crump \& C.\ J.\ Mode} (1968b)
A general age-dependent branching process, II.
{\it J.\ Math.\ Anal.\ Appl.\/} {\bf 25},  8--17. 

\bibitem{HG91}
{\sc I.\ Hanski \&  M.\ Gilpin} (1991)
Metapopulation dynamics: brief history and conceptual domain.
{\it Biol.\ J.\ Linnean Soc.\/} {\bf 42}, 3--16. 

\bibitem{Jagers68}
{\sc P.\ Jagers} (1968)
Age-dependent branching processes allowing immigration. 
{\it Teor.\ Verojatnost.\ i Primenen\/} {\bf 13}, 230--242. 

\bibitem{Jagers75}
{\sc P.\ Jagers} (1975) 
{\it Branching processes with biological applications.}
 Wiley, New York.

\bibitem{kr93}  
{\sc M.~Kretzschmar} (1993) 
Comparison of an infinite dimensional model for para\-sitic diseases with a 
related 2-dimensional system. 
{\em J.\ Math.\ Analysis Applics} {\bf 176}, 235--260. 
MR1222167

\bibitem{k70}
{\sc T.\ G.\ Kurtz} (1970) 
Solutions of ordinary differential equations as limits of pure jump Markov processes.
{\it J.\ Appl.\ Probab.\/}~{\bf 7}, 49--58. 
MR0254917

\bibitem{k71}
{\sc T.\ G.\ Kurtz} (1971) 
Limit theorems for sequences of jump Markov processes approximating
ordinary differential processes.
{\it J.\ Appl.\ Probab.\/}~{\bf 8}, 344--356. 
MR0287609

\bibitem{CL90}
{\sc C.\ L\'eonard} (1990)
Some epidemic systems are long range interacting particle systems.
In: {\it Stochastic Processes in Epidemic Theory\/}, Eds J.\ P.\ Gabriel, C.\ Lef\`evre and P.\ Picard,
Lecture Notes in Biomathematics 86, pp.\ 170--183. Springer, New York. 


\bibitem{lev69}
{\sc R.\ Levins}  (1969)
Some demographic and genetic consequences of environmental heterogeneity
for biological control. 
{\it Bull.\ Entomol.\ Soc.\ Amer.\/} {\bf 15}, 237--240.


\bibitem{lu01a} {\sc C.\ J.\ Luchsinger} (2001a) 
Stochastic models of a parasitic infection, exhibiting three basic reproduction ratios. 
{\it J.\ Math.\ Biol.\/} {\bf 42}, 532--554. 
MR1845591
 
\bibitem{lu01b} {\sc C.\ J.\ Luchsinger} (2001b) 
Approximating the long term behaviour of a model for parasitic infection. 
{\it J.\ Math.\ Biol.\/} {\bf 42}, 555--581.
MR1845592 

\bibitem{MetzGyll} {\sc J.\ A.\ J.\ Metz \& M.\ Gyllenberg} (2001)
How should we define fitness in structured metapopulation models? Including an application 
to the calculation of evolutionarily stable dispersal strategies.
{\it Proc.\ Roy.\ Soc.\/}~B {\bf 268}, 499--508.




\end{thebibliography}
\end{document}